\documentclass[12pt]{amsart}
\usepackage{amscd,amssymb,amsthm,amsmath,amssymb,textcomp,supertabular,longtable,enumerate,rotating,mathrsfs,mathtools,hyperref,latexsym,amscd,amsbsy,mathrsfs,biblatex,xcolor, csquotes}
\usepackage{pdflscape}
\usepackage[matrix,arrow,curve]{xy}
\usepackage{tikz-cd}
\newtheorem{theorem}[equation]{Theorem}

\newtheorem{proposition}[equation]{Proposition}
\newtheorem{lemma}[equation]{Lemma}
\newtheorem{corollary}[equation]{Corollary}

\newtheorem{maintheorem}{Main Theorem}

\theoremstyle{definition}
\newtheorem{example}[equation]{Example}
\newtheorem{definition}[equation]{Definition}

\theoremstyle{remark}
\newtheorem{remark}[equation]{Remark}

\makeatletter\@addtoreset{equation}{section} \makeatother

\swapnumbers

\newtheoremstyle{dotless}{}{}{\rm}{}{\sc}{}{ }{}

\theoremstyle{dotless}

\newcommand{\sq}{\mathfrak S_4}
\newcommand{\pgl}{\mathrm{PGL}}
\newcommand{\gl}{\mathrm{GL}}
\newcommand{\specl}{\mathrm{SL}}
\newcommand{\pp}{\mathbb P^2}
\newcommand{\dc}{\mathbb C}

\newcommand{\dz}{\mathbb Z}

\newcommand{\aut}{\mathrm{Aut}}

\newcommand{\pl}{\mathbb P^1}

\newcommand{\dt}{\mathbb{T}}
\newcommand{\pic}{\mathrm{Pic}}
\newcommand{\ord}{\mathrm{o}}

\newcommand{\ac}{\mathfrak A_5}
\newcommand{\aq}{\mathfrak A_4}
\newcommand{\dps}{(x:y:z)\times(u:v:w)}
\newcommand{\st}{\mathfrak S_3}
\newcommand{\scinq}{\mathfrak S_5}
\newcommand{\id}{\mathrm{id}}

\newcommand{\dih}{\mathrm{D}}

\author{Antoine Pinardin}

\title{$G$-Solid rational surfaces}

\address{\emph{Antoine Pinardin}
\newline
\textnormal{School of Mathematics, The University of Edinburgh, Edinburgh EH9 3JZ, UK.}
\newline
\textnormal{\texttt{antoine.pinardin@ed.ac.uk}}}

\begin{document}
\maketitle
\begin{abstract}
    We classify $G$-solid rational surfaces over the field of complex numbers for finite group actions.
\end{abstract}

\section{Introduction}

We are interested in the equivariant birational geometry of rational surfaces over the field of complex numbers for finite group actions. Let $S$ be a rational surface, and $G$ be a finite group acting faithfully and biregularly on $S$. Denote by $\rho^G(S)$ the rank of the $G$-invariant part of the Picard group of $S$. The $G$-equivariant Minimal Model Program applied to a resolution of singularities of $S$ implies that $S$ is $G$-birational to a \textit{$G$-Mori fibre space}, i.e. a $G$-surface in one of the following two cases:
\begin{itemize}
	\item A $G$-del Pezzo surface, namely a del Pezzo surface $S$ such that $-K_S$ is ample and $\rho^G(S)=1$,
	\item A $G$-conic bundle, i.e. there is a $G$-equivariant morphism $S\rightarrow\pl$ with general fibre isomorphic to $\pl$, and such that $\rho^G(S)=2$.
\end{itemize}
We say that $S$ is \textit{$G$-solid} if it is not $G$-birational to any $G$-conic bundle.

\begin{maintheorem}
    Let $S$ be a $G$-del Pezzo surface of degree $d=K_S^2$. Then $S$ is $G$-solid if and only if:
    \begin{itemize}
        \item $d\le 3$.
        \item $d=4$ and $G$ does not fix a point on $S$ in general position.
        \item $d=5$ and $G$ is not isomorphic to $\dz_5$ or $\dih_5$.
        \item $d=6$ and $G$ is not isomorphic to $\dz_6$, $\st$, or $\dih_6$.
        \item $d=8$, $S\cong\pl\times\pl$, and, up to conjugation in $\aut(S)$, either
        \begin{itemize}
            \item $G$ has a subgroup isomorphic to $\aq$,
            \item or $G_4\subset G$ and $G\not\subset G_{16}$, for two specific groups $G_4$ and $G_{16}$.
        \end{itemize}
        \item $S\cong\pp$, the group $G$ does not fix a point on $S$, and is not isomorphic to $\sq$ or $\aq$.
    \end{itemize}
\end{maintheorem}
A complete answer was already given for $K_S\le4$ by Segre in \cite{segre1943note}, Manin in \cite{manin1966rational}, Das Dores-Mauri in \cite{das2019g}, and Dolgachev-Iskovskikh in \cite{dois}. We will quickly recall these results in section \ref{formalism}. For $S=\pp$, we also prove the following result.

\begin{theorem}
    Let $G$ be a finite subgroup of $\aut(\pp)$ isomorphic to $\aq$ or $\sq$. The only $G$-Mori fibre spaces $G$-birational to $\pp$ are $\pp$ and a del Pezzo surface of degree $5$ with a $G$-conic bundle structure.
\end{theorem}

In particular, for $G\cong\aq$ and for $G\cong\sq$, the projective plane is not $G$-birational to any Hirzebruch surface $\mathbb F_n$.

A recent motivation for our work is given by Tschinkel, Yang, and Zhang in \cite{tschinkel2023equivariant}. Both $G$-solid surfaces and surfaces which are not $G$-birational to a Hirzebruch surface are classes of divisors on threefolds which give rise to \textit{Incompressible Divisorial Symbols}, a modern tool used in the formalism of Burnside groups to distinguish birational types of group actions. An example of applications of these techniques by Cheltsov, Tshinkel, and Zhang can be found in \cite{cheltsov2023equivariant}.

Finally, let us point out that the classification of $G$-solid Fano varieties is widely open starting from dimension $3$. Some work has been achieved in the case of toric Fano threefolds, see for example \cite{cheltsov2023toric}, \cite{cheltsov2023birational}, and \cite{cheltsov2023equivariant}.

\section{Formalism and existing results}\label{formalism}
\subsection{The Sarkisov program}
We recall a powerful tool for studying the $G$-equivariant birational geometry of rational surfaces. Any $G$-birational map between $G$-Mori fibre spaces splits into a sequence of elementary $G$-birational maps, called \textit{$G$-Sarkisov links}.

\begin{definition}
    A $G$-\textit{Sarkisov link}, from $S$ to $S'$, is a commutative diagram of one of the following forms.
\end{definition}
\begin{itemize}
    \item Links of type \MakeUppercase{\romannumeral 1}:
    $$\xymatrix
            {
                 &S'\ar[ld]_\sigma\ar[d]\\
                 S&\pl
            }$$
    where $S$ is a $G$-del Pezzo surface, $S'\rightarrow\pl$ is a $G$-conic bundle, and $\sigma$ is the blow-up of a $G$-orbit.
    \item Links of type \MakeUppercase{\romannumeral 2}:
    $$\xymatrix
            {
                 &Z\ar[ld]_{\sigma}\ar[rd]^{\sigma'}&\\
                 S\ar[d]_{\pi}\ar@{-->}[rr]&&S'\ar[d]^{\pi'}\\
                 B\ar@{=}[rr]&&B'
            }$$
    where either $B=\{\mathrm{pt}\}$, in which case $S$ and $S'$ are $G$-del Pezzo surfaces, or $B\cong\pl$, and $\pi,\pi'$ are $G$-conic bundles. The maps $\sigma$ and $\sigma'$ are blow-ups of $G$-orbits.
    \item Links of type \MakeUppercase{\romannumeral 3}. They are the inverses of the links of type \MakeUppercase{\romannumeral 1}.
    \item Links of type \MakeUppercase{\romannumeral 4}. Here, $S=S'$ is a $G$-conic bundle, and any such link is the choice of one conic bundle structure on $S$ among the possible two.
\end{itemize}

In the following theorem, we gather the two main results about the Sarkisov program for surfaces, both taken from \cite{iskovskikh1996factorization}.

\begin{theorem}
        Any $G$-birational map $S\dashrightarrow S'$ between $G$-Mori fibre spaces splits into finitely many $G$-links. Moreover, the $G$-links between $G$-Mori fibre spaces are classified fully classified in \cite{iskovskikh1996factorization}.
\end{theorem}

We will not recall this classification here, but will carefully mention what we use from it throughout the present note. Let us mention that a $G$-link of type \MakeUppercase{\romannumeral 1} or \MakeUppercase{\romannumeral 2} starting from a surface $S$ is always centered at an orbit of length lower than $d=K_S^2$.

\subsection{Existing results}

The case of $G$-del Pezzo surfaces $S$ of canonical degree $d\le3$ is often referred to as the Segre-Manin Theorem. Recall that a $G$-del Pezzo surface $S$ is called $G$-rigid if any $G$-birational map starting from $S$ leads to a $G$-isomorphic surface. The surface $S$ is called $G$-superrigid if any $G$-birational map starting from $S$ is a $G$-isomorphism.
\begin{proposition}
    Let $S$ be a smooth del Pezzo surface, and $G$ be a finite subgroup of $\aut(S)$ such that $\rho^G(S)=1$.
    \begin{itemize}
        \item If $K_S^2=1$, $S$ is $G$-superrigid.
        \item IF $K_S^2=2$ or $K_S^2=3$, $S$ is $G$-rigid.
    \end{itemize}
\end{proposition}

In \cite{das2019g}, Das Dores and Mauri also classified finite groups $G$ for which a del Pezzo surface of degree $2$ or $3$ is $G$-superrigid.

For $K_S^2=4$, the result is essentially a corollary of the classification of $G$-links in \cite{iskovskikh1996factorization}.

\begin{proposition}
    Let $G$ be a subgroup of $\aut(S)$ such that $\rho^G(S)=1$, where $S$ is a del Pezzo surface of degree $4$. The following conditions are equivalent.
    \begin{itemize}
        \item $S$ is $G$-rigid,
        \item $S$ is $G$-solid,
        \item $G$ does not fix a point on $S$ outside of the $(-1)$-curves.
    \end{itemize}
\end{proposition}
\begin{proof}
    Using the classification of Sarkisov links in \cite{iskovskikh1996factorization}, we see that there are three possible links starting from $S$. The first one is the blow-up of a point not lying on a $(-1)$-curve, leading to a $G$-conic bundle. The second one is a Geiser involution, centered at a an orbit of length $2$, and the last one is a Bertini involution, centered at an orbit of length $3$. But as mentioned in \cite{egor}, Geiser and Bertini involutions lead to $G$-isomorphic surfaces.
\end{proof}

\section{Del Pezzo surfaces of degree $5$}

Up to isomorphism, there is only one smooth del Pezzo surface $S$ of degree $5$, given by the blow-up of $\pp$ in four points in general position. Its group of automorphisms is isomorphic to $\scinq$, nicely described for example in \cite{theseblanc}.

\begin{lemma}[\cite{dois}]
    Let $G\subset\aut(S)$. Then $\rho^G(S)=1$ if and only if $G$ is isomorphic to $\scinq$, $\ac$, $\dih_5$, $\dz_5$, or $F_5=\dz_5\rtimes\dz_4$\footnote{This is the group of GAP ID $(20,3)$.}.
\end{lemma}

Let us recall the groups for which the $G$-solidity of $S$ is known. The cases of $\ac$ and $\scinq$ were solved by Cheltsov, and $G=F_5$ by Wolter.

\begin{proposition}[\cite{cheltsov2007log}, \cite{cheltsov2014two}]
	If $G\subset\aut(S)$ is isomorphic to $\ac$ or $\scinq$, then $S$ is $G$-superrigid.
\end{proposition}

\begin{proposition}[\cite{wolt}]
	If $G\subset\aut(S)$ is isomorphic to $F_5$, then $S$ is $G$-solid.
\end{proposition}

What remains to study is the $G$-solidity of $S$ for $G$ isomorphic to $\dz_5$ or $\dih_5$. We will manage to avoid studying the $G$-orbits on $S$, and only need to use the $G$-birational geometry of $\pp$.

\begin{proposition}
    If $G$ is isomorphic to $\dz_5$ or $\dih_5$, then $S$ is not $G$-solid.
\end{proposition}

\begin{proof}
    Consider the matrices $M=\begin{pmatrix}1&0&0\\0&\mu_5&0\\0&0&\mu_5^{-1}\end{pmatrix}\in\pgl_3(\dc)$, where $\mu_5$ is a primitive fifth root of the unity, and $N=\begin{pmatrix}1&0&0\\0&0&1\\0&1&0\end{pmatrix}$. We have $H:=\left<M\right>\cong\dz_5$, and $H':=\left<M,N\right>\cong \dih_5$, under the action of both groups, the point $(1:1:1)$ have an orbit of length $5$ in general position. Blowing-up this orbit and contracting the proper transform of the conic passing through the five points gives a $G$-link from $\pp$ to a del Pezzo surface of degree $5$, for any $G\in\{H,H'\}$. On the other hand, for any $G\in\{H,H'\}$, the point $(1:0:0)\in\pp$ is fixed under the action of $G$. Hence, we can blow it up and get a $G$-link to the Hirzebruch surface $\mathbb F_1$, with a $G$-conic bundle structure. Since $\dz_5$ and $\dih_5$ are unique in $\scinq$ up to conjugacy, we conclude that $S$ is not $G$-solid for any subgroup $G$ of $\aut(S)$ isomorphic either to $\dz_5$ or to $\dih_5$.
\end{proof}

\section{Del Pezzo surfaces of degree $8$}\label{dp8}
Let $S=\pl\times\pl$. Recall that this is the only del Pezzo surface of degree $8$ we have to study, since the blow-up of $\pp$ at a point cannot be a $G$-del Pezzo surface. The automorphism group of $S$ is isomorphic to $(\pgl_2(\dc)\times\pgl_2(\dc))\rtimes\dz_2$, where $\dz_2$ acts on the direct product $\pgl_2(\dc)\times\pgl_2(\dc)$ by permuting its factors. We will use affine coordinates $x$ and $y$ and have the two $\pgl_2(\dc)$ components act on them, respectively. The action of $\dz_2$ is given by the permutation of $x$ and $y$. For example, an automorphism written $(x_0:x_1)\times(y_0:y_1)\mapsto(y_1:y_0)\times(x_0:x_1)$ will be denoted as $(\frac{1}{y},x)$. To begin with, we will give an example of a group $G\subset\aut(S)$ such that the surface $S$ is not $G$-solid.

\begin{example}\label{2fixed}
    Let $\pi_1$ and $\pi_2$ be the two canonical projections from $S$ to $\pl$. If two points $P$ and $Q$ are such that $\pi_i(P)\ne\pi_i(Q)$ for $i=1$ and $i=2$, we say that they are \textit{in general position}. Assume that a subgroup $G\subset\aut(S)$ fixes two points in general position. Then there is a $G$-birational map from $S$ to the $G$-conic bundle $\mathbb F_1$, which decomposes into two $G$-links as follows:
    \[
        \xymatrix
            {
                 &Z_7\ar[ld]_{\sigma_1}\ar[rd]^{\sigma_2}&&\mathbb F_1\ar[d]\ar[ld]_{\sigma_3}\\
                 S\ar@{-->}[rr]&&\pp&\pl
            }
        \]
        where $\sigma_1$, $\sigma_2$ and $\sigma_3$ are blow-ups at a point. Here is a list of explicit cases in which it happens:
        \begin{itemize}
            \item If $G=\left<s,t_n\right>$, with $s=(y,x)$, and $t=(\mu_nx,\mu_n^{-1}y)$, then $G\cong \dih_n$.
            \item If $G=\left<\sigma,\tau_n\right>$, with $\sigma=(y,-x)$ and $\tau_n=(ix,iy)$, then $G\cong Q_8$.
        \end{itemize}
\end{example}

\begin{remark}\label{glinksp1p1}
        Let $\phi\colon S\dashrightarrow S'$ be a $G$-link centered at $k$ points. Using the classification of Sarkisov links in \cite{iskovskikh1996factorization}, we see that, if $k\ge6$, then $\phi$ is either a Bertini or a Geiser involution. Once again, as mentioned in \cite{egor}, such link leads to a $G$-isomorphic surface. It follows that if $S'$ is not $G$-isomorphic to $S$, then $\phi$ is of one of the following forms.
    \begin{center}
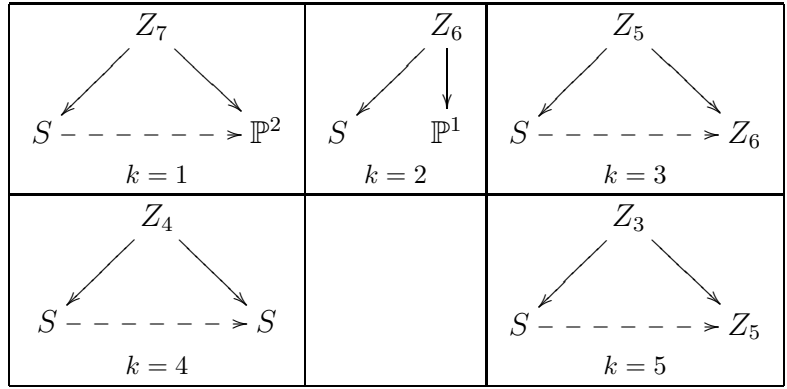
        \begin{tabular}{|c|c|c|}
            \hline
            $\xymatrix
                {
                     &Z_7\ar[ld]\ar[rd]&\\
                     S\ar@{-->}[rr]&&\pp
                }$&$\xymatrix
                {
                     &Z_6\ar[ld]\ar[d]\\
                     S&\pl
                }$&$\xymatrix
                {
                     &Z_5\ar[ld]\ar[rd]&\\
                     S\ar@{-->}[rr]&&Z_6
                }$\\
    \footnotesize{$k=1$}&\footnotesize{$k=2$}&\footnotesize{$k=3$}\\
            \hline
            $\xymatrix
                {
                     &Z_4\ar[ld]\ar[rd]&\\
                     S\ar@{-->}[rr]&&S
                }$&&$\xymatrix
                {
                     &Z_3\ar[ld]\ar[rd]&\\
                     S\ar@{-->}[rr]&&Z_5
                }$\\
            \footnotesize{$k=4$}&&\footnotesize{$k=5$}\\
            \hline
        \end{tabular}
    \end{center}
    A surface denoted by $Z_k$ is a del Pezzo surface of degree $k$. We deduce that if $G$ does not have any orbit of length $k\le5$, then $S$ is $G$-rigid. One can mention that every pair of points in the center of one of the links of Remark \ref{glinksp1p1} must be in general position, as defined in Example \ref{2fixed}.
\end{remark}

\subsection{Toric subgroups of $\aut(\pl\times\pl)$}

We can embed $(\dc^*)^2$ as a dense torus in $S$ by the map $\iota\colon(\dc^*)^2\hookrightarrow S$, $(a,b)\mapsto(1:a)\times(1:b)$, whose image will be called $\mathfrak T$. Moreover, any dense torus in $S$ is equal to $\mathfrak T$ up to an automorphism of $S$. The action of $(\dc^*)^2$ on itself by translation extends to a faithful action on the whole variety $S$, by identifying $(\dc^*)^2$ to the subgroup $\dt=\{(ax,by),a,b\in\dc^*\}$ of $\aut(S)$.

\begin{lemma}\label{es1}
    There is an exact sequence
\[\begin{tikzcd}
	1 & {\mathbb{T}} & {N(\mathbb{T})} & {\dih_4} & {1.}
	\arrow[from=1-1, to=1-2]
	\arrow[from=1-2, to=1-3]
	\arrow["w", from=1-3, to=1-4]
	\arrow[from=1-4, to=1-5]
\end{tikzcd}\label{es}\]
\end{lemma}
\begin{proof}
    First, notice that $N(\dt)$ leaves $\mathfrak T$ invariant. Indeed, the normalizer of $\dt$ permutes the $\dt$-orbits, and $\mathfrak T$ is the only one which is dense in $S$. The complement of $\mathfrak T$ in $S$ is the divisor $C:=\pi_1^{-1}(1:0)+\pi_1^{-1}(0:1)+\pi_2^{-1}(1:0)+\pi_2^{-1}(0:1)$, where $\pi_1$ (resp. $\pi_2$) is the canonical projection from $\pl\times\pl$ to the left (resp. right) factor $\pl$. Intersection number being preserved by automorphisms, any element of $N(\mathbb T)$ induces a symmetry of the square formed by $C$, thus giving a group homomorphism $w\colon N(\mathbb T)\rightarrow \dih_4$. Finally, the group $\dt$ is the set of automorphisms which preserve each irreducible curve of the divisor $C$. In other words, the kernel of $w$ is $\dt$. Moreover, $w$ is surjective, hence the exact sequence.
\end{proof}

\begin{remark}\label{postes1}
    Notice that $N(\dt)$ is exactly the set of automorphisms preserving the square $C$. This will be useful later on for proving that a subgroup $G$ of $\aut(S)$ is contained in $N(\dt)$.
\end{remark}

Let $G$ be a subgroup of $N(\dt)$, and $T=G\cap\dt$. The restriction of $w$ to $G$ induces the exact sequence \[\begin{tikzcd}
	1 & {T} & {G} & {W} & {1,}
	\arrow[from=1-1, to=1-2]
	\arrow[from=1-2, to=1-3]
	\arrow["w", from=1-3, to=1-4]
	\arrow[from=1-4, to=1-5]
\end{tikzcd}\]for some subgroup $W$ of $\dih_4$.

We will call \textit{toric} a subgroup of $\aut(S)$ conjugated to a subgroup of $N(\dt)$. Consider the automorphisms $r=(\frac{1}{y},x)$, and $s=(y,x)$. Both belong to $N(\dt)$, and they generate a group isomorphic to $\dih_4$, on which $w$ restricts to an isomorphism. In \cite{cheltsov2023toric}, the authors mention without proof that if $G$ is a subgroup of $N(\dt)$ such that $\rho^G(S)=1$, then its image $W$ in $\dih_4$ must contain $\dz_4$. We prove this result here, and state it in a slightly stronger way, adding that an element mapped onto $w(r)$, hence generating $\dz_4$ in $\dih_4$, is equal to $r$ up to conjugation by an element of the torus.

\begin{lemma}
    Let $G$ be a toric subgroup of $\aut(S)$ such that $\rho^G(S)=1$. If $S$ is $G$ solid, then $G$ is conjugated by an element of $\dt$ to a subgroup of $N(\dt)$ containing $r$.
\end{lemma}
\begin{proof}
    Since $G$ is toric, we may assume that it is a subgroup of $N(\dt)$. The orbits of the vertices of the square $C$ under the action of $\dih_4$ are of length $1$, $2$, or $4$. If one of them is fixed, the opposite vertex must be fixed as well, and we get the situation of Example \ref{2fixed}. In particular, the surface $S$ is not $G$-solid. If no vertex is fixed, then either the vertices are on the same orbit, or they form two orbits of length $2$. Since $\rho^G(S)=1$, we cannot have two orbits of length $2$, each of them consisting of two consecutive vertices of $C$. So these two orbits must be formed by opposite vertices, and we deduce that there are two orbits of length $2$ in general position. We can blow-up one of them to get a $G$-link to a $G$-conic bundle. Hence, the only possibility for $S$ to be $G$-solid is that all the vertices are on the same orbit. In this case, since $\rho^G(S)=1$, the vertices must be cyclically permuted, so $w(r)\in W$. It remains to point out that if an element $g\in\aut(S)$ satisfies $w(g)=w(r)$, then $g$ is conjugated to $r$ in $N(\dt)$. Indeed, such automorphism $g$ is of the form $(\frac{a}{y},bx)$, for some $a,b\in\dc^*$. Let $t=(kx,ly)\in\dt$, with $k,l\in\dc^*$. If $l^2=\frac{1}{ab}$ and $k=bl$, then we have $tgt^{-1}=r$.
\end{proof}

\begin{lemma}
	Let $G$ be a subgroup of $N(\dt)$ containing $r$, and $T=G\cap\dt$. If $|T|>5$, then $S$ is $G$-rigid.
\end{lemma}
\begin{proof}
	The $G$-orbits of the points outside of the square $C$ are of length at least $T$. According to Remark \ref{glinksp1p1}, there is no link centered in such orbit leading to a non-$G$-isomorphic surface. Moreover, the orbit of a point lying in the square $C$ is not in general position.
\end{proof}

We will now start the exhaustive study of subgroups $G$ of $N(\dt)$ such that $\rho^G(S)=1$, where $|G\cap\dt|\le5$ and $r\in\dt$. Let $G$ be such subgroup of $\aut(S)$. First, we can take off the cases where $G\cong\dz_3$ or $G\cong\dz_4$.

\begin{lemma}
    Let $G$ be a subgroup of $N(\dt)$ containing $r$. Then the toric part $T=G\cap\dt$ cannot be isomorphic to $\dz_3$ or to $\dz_4$.
\end{lemma}
\begin{proof}
    Assume that $T$ is isomorphic to $\dz_3$ (resp. $\dz_4$). Then $T$ is generated by an automorphism $t\in\dt$ of the form $(\mu_nx,\mu_n^ky)$ or $(\mu_n^kx,\mu_ny)$, where $n=3$ (resp. $4$), and $\mu_n$ is a primitive $n$-th root of the unity. But by Lemma \ref{conjr}, we may assume that $G$ contains the element $r=(\frac{1}{y},x)$, acting on $T$ by conjugation. Since $(r(\mu_n^kx,\mu_ny)r^{-1})^{-1}=(\mu_nx,\mu_n^{-k}y)$, we may assume that $t$ is of the first form, namely we have $T=\left<(\mu_nx,\mu_n^ky)\right>$. The automorphism $rtr^{-1}=(\mu_n^{-k}x,\mu_ny)\in T$ must belong to $T$, so must be a power $(\mu_nx,\mu_n^ky)$. It is impossible for $n=3$ or $n=4$.
\end{proof}

\begin{proposition}
    If $T$ is trivial, then $S$ is not $G$-solid. The options for $G$ are the following.
    \begin{itemize}
        \item A group isomorphic to $\dz_4$, generated by $r=(\frac{1}{y},x)$,
        \item A group isomorphic to $\dih_4$, generated by $r$ and $s=(y,x)$,
        \item A group isomorphic to $\dih_4$, generated by $r$ and $(-y,-x)$.
    \end{itemize}
\end{proposition}
\begin{proof}
	The toric part $T$ is trivial, so the map $w$ restricts to an isomorphism on $G$. Since $r\in G$, we either have $G\cong\dz_4$, or $G\cong \dih_4$. Assume we are in the latter case. Then $G$ is generated by $r$, and an element $h$ such that $w(h)=w(s)$, or in other words such that $h=ts$, for some $t=(ax,by)\in\dt\cong(\dc^*)^2$. Since $w$ restricts to an isomorphism on $G$, we can write $hrh=w^{-1}(srs)=w^{-1}(r^{-1})=r^{-1}$. It yields $r^{-1}=(a^2y,\frac{1}{x})$, so that $a^2=1$. Finally, since $h$ is of order $2$, we get $ab=1$. So the only two options for $h$ are $h=s$, or $h=(-y,-x)$. If $G$ is generated by $r$ and $s$, the points $(1,1)$ and $(-1,-1)$ are fixed by the action, in which case $S$ is not $G$-solid, as in example \ref{2fixed}. If $G$ is generated by $r$ and $(-y,-x)$, the points $(1,1)$ and $(-1,-1)$ form an orbit of length $2$. We can $G$-equivariantly blow-up these points and obtain a $G$-link to a $G$-conic bundle. Finally, if $\dz_4\cong G=\left<r\right>$, then $G$ is a subgroup of $\left<r,s\right>$, so that the surface $S$ is not $G$-solid.
\end{proof}

\begin{proposition}
    If $T\cong \dz_2$, there are the following possibilities and only them:
    \begin{itemize}
        \item $G=\left<t,r\right>\cong\dz_4\times\dz_2$, with $t=(-x,-y)$. The surface $S$ is not $G$-solid.
        \item $G=\left<t,r,s\right>\cong\dz_2\times \dih_4$. The surface $S$ is not $G$-solid.
        \item $G=\left<t,r,h\right>\cong\dz_2^2\rtimes\dz_4$, with $t$ and $g$ as above, and $h=(-y,x)$.The surface $S$ is $G$-solid.
    \end{itemize}
\end{proposition}
\begin{proof}
Assume that $w(G)=\dih_4$. Then $G$ is generated by $r$, an element $t\in T$, and an element $h$ such that $w(h)=(ay,bx)$, for some $a,b\in\dc^*$. Since $T$ is of order $2$, the element $t$ is either $(-x,-y),(-x,y)$, or $(x,-y)$. But the relation $rtr^{-1}=t$ implies that the only possibility is $t=(-x,-y)$. The order of $h$ can only be $2$ or $4$. If $\ord(h)=2$, then $a=b^{-1}$. Since $w(hrh)=w(r^{-1})$, we have $hrh=\tau r^{-1}=\tau\cdot(y,\frac{1}{x})$, for some $\tau\in T$, i.e. $\tau=\id$ or $\tau=t$. But $hrh=\tau\cdot(a^2y,\frac{1}{x})$, which implies $a=b=\pm1$. But since $t\in G$, we may assume that $a=b=1$. The points $P_1=(1:1)\times(1:1)$ and $P_2=(1:-1)\times(1:-1)$ form an orbit of length $2$, and blowing it up gives a $G$-link to a $G$-conic bundle.

Assume $\ord(h)=4$. Since $h^2=(abx,aby)$, this implies that $a^2b^2=1$ and $ab\ne1$, so that $ab=-1$ and $h=(ax,-\frac{y}{a})$. Using again the fact that $r^{-1}=\tau hrh^{-1}=\tau\cdot(a^2y,\frac{1}{x})$, for some $\tau\in T$, we deduce that $h=(-y,x)$ or $h=(y,-x)$. Since these possibilities only differ by $t$, we may assume that $h=(-y,x)$. The group $G$ is generated by $t$, $r$, and $h$, and is isomorphic to $G\cong\dz_2^2\rtimes\dz_4$. There is no $G$-orbit of length $l\le5$, hence $S$ is $G$-solid.

Finally, if $w(G)=\dz_4$, then $\dz_2\times\dz_4\cong G=\left<t,r\right>\subset\left<t,r,s\right>$, so that the surface $S$ is not $G$-solid, as in the case where $G=\left<t,r,s\right>.$
\end{proof}

\begin{proposition}
    If $T\cong \dz_2^2$, then $S$ is $G$-solid. Moreover, there are the following possibilities for $G$, and only them:
    \begin{itemize}
        \item $G\cong\dz_2^2\rtimes\dz_4$, generated by $r$ and $t=(-x,y)$,
        \item $G\cong\dz_2^4\rtimes\dz_2$, generated by $r$, $t=(-x,y)$, and $s=(y,x)$,\footnote{This semidirect product is the group of GAP ID $(32,27)$.}
        \item $G\cong\dz_2^3\rtimes\dz_4$, generated by $r$, $t=(-x,y)$, and $h=(iy,ix)$.\footnote{This semidirect product is the group of GAP ID $(32,6)$.}
    \end{itemize}
\end{proposition}
\begin{proof}
    Assume that $w(G)\cong\dz_4$, so that $G=\left<T,r\right>$. Since $T\subset(\dc^*)^2$ is isomorphic to $\dz_2^2$, it is generated by $t_1=(-x,y)$ and $t_2=(x,-y)$. The group $G$ is isomorphic to $\dz_2^2\rtimes\dz_4$. Notice that since $rt_1r^{-1}=t_2$, the group $G$ is generated by $r$ and $t_1$. There is no $G$-orbit of length $l\le5$. Hence, the surface $S$ is $G$-solid. If $w(G)=\dih_4$, the group $G$ contains $\left<r,t_1\right>$. Hence the surface $S$ is also $G$-solid.

    Let us complete the list of groups whose toric part is isomorphic to $\dz_2^2$. Assume that $w(G)=\dih_4$, so that there exists an element $h\in G$ such that $w(h)=w(s)$. The order of $h$ is either $2$ or $4$. If $\ord(h)=2$, we get $h=(ay,\frac{x}{a})$. But $r^{-1}=\tau hrh^{-1}=\tau\cdot(a^2y,\frac{1}{x})$, so $\tau=\mathrm{id}$ and $a=\pm1$, or $\tau=t_1=(-x,y)$ and $a=\pm i$. Up to composition by an element of $T$, we get $h=s=(y,x)$ or $h=(-ix,iy)$. In the first case, we have that $G=\left<t_1,t,h\right>$ is isomorphic to a semidirect product of the form $\dz_2^3\rtimes\dz_4$, and in the second case to a semidirect product of the form $\dz_2^3\rtimes\dz_4$. If $\ord(h)=4$, we get $h=(ay,-\frac{x}{a})$, which only differs from the previous case by a an element of $T$, so $G$ is still either $\left<t_1,r,s\right>$, or $\left<t_1,r,(-ix,iy)\right>$.
\end{proof}

The only remaining possibility is $T\cong\dz_5$. We will use the results of \cite{wolt} about the $G$-solidity of the del Pezzo surface of degree $5$.

\begin{proposition}
    If $T\cong\dz_5$, then $G$ is isomorphic to the Fröbenius group $F_5\cong\dz_5\rtimes\dz_4$, generated by $r$ and $(\mu_5x,\mu_5^ly)$ for $l=2$ or $l=3$, where $\mu_5$ is a primitive fifth root of the unity.
\end{proposition}
\begin{proof}
    Assume $T\cong\dz_5$. It is generated by an element $t$ of the form $t=(\mu_5x,\mu_5^ly)$ or $t=(\mu_5^kx,\mu_5y)$. Since $r(\mu_5^kx,\mu_5y)^{-1}r^{-1}=(\mu_5x,\mu_5^{-k}y)$, we may assume that $t=(\mu_5x,\mu_5^ly)$. But then $rtr^{-1}=(\mu_5x^{-l},\mu_5y)$, and this element is in $\left<t\right>$ if and only if $l=2$ or $l=3$. Assume that there is an element $h\in G$ such that $w(h)=w(s)$. Such element is of the form $(ay,bx)$, for some $a,b\in\dc^*$. But $hth^{-1}=(\mu_5^lx,\mu_5y)$, which is not in $T$. Since this subgroup is normal in $G$, we get a contradiction. Hence, $G$ is generated by $(\mu_5x,\mu_5^ly)$ and $r$, with $l=2$ or $l=3$. In both cases, $G$ is isomorphic to the Fröbenius group $F_5\cong\dz_5\rtimes\dz_4$. The points outside of $C$ have orbits of length at least $5$, so that the only possible $G$-Sarkisov does not lead to a $G$-isomorphic surface leads to a del Pezzo surface of degree $5$ with invariant Picard rank $1$. This surface is $G$-solid, according to \cite{wolt}.
\end{proof}

\subsection{Non-toric subgroups of $\aut(\pl\times\pl)$}

\begin{proposition}\label{sgpsautp1p1}
    Let $G$ be a subgroup of $\aut(S)$ such that $\rho^G(S)=1$. Up to conjugation of $G$ in $\aut(S)$, there is an exact sequence of the form
    \[\begin{tikzcd}
    	1 & {H\times_DH} & {G} & {\dz_2} & {1.}
    	\arrow[from=1-1, to=1-2]
    	\arrow[from=1-2, to=1-3]
    	\arrow["\delta", from=1-3, to=1-4]
    	\arrow[from=1-4, to=1-5]
    \end{tikzcd}\]\label{es2}
where $H$ is the projection of $G$ onto the first and second $\pgl_2(\dc)$ factor in $(\pgl_2(\dc)\times\pgl_2(\dc))\rtimes\dz_2$.
\end{proposition}
Before proving the above result, let us recall that the \textit{fibre product} of two groups $H$ and $H'$, over a group $D$, for the surjective group morphisms $\phi\colon H\rightarrow D$ and $\psi\colon H'\rightarrow D$, is the subgroup $\{(h,h'),\phi(h)=\psi(h')\}$ of $H\times H'$. In some of our proofs, we will have to be particularly careful about the surjective group morphisms $\phi$ and $\psi$, in which case we will use the notation $H\times_{D,\phi,\psi}H$, instead of $H\times_DH$. The following Lemma gives the structure of the subgroups of a direct product, in terms of fibre products.

\begin{lemma}[Goursat's Lemma]\label{goursat}
    Let $H$ and $H'$ be groups, and $p_1\colon H\times H'\rightarrow H$, $p_2\colon H\times H'\rightarrow H'$ be the two canonical projections. Let $R$ be a subgroup of $H\times H'$ such that $p_1(R)=H$ and $p_2(R)=H'$. Then $R$ is a fibre product of the form $H\times_DH'$.
\end{lemma}
We can now prove Proposition \ref{sgpsautp1p1}.
\begin{proof}[Proof of \ref{sgpsautp1p1}]
    Define the group morphism $\delta\colon\aut(S)\rightarrow\dz_2$ which sends an element $g$ to $1$ if and only if $g$ swaps the rulings of $\pl\times\pl$. Since $\rho^G(S)=1$, there must be such element in $G$, hence the surjectivity of the restriction of $\delta$ to $G$. The kernel of $\delta$ in $\aut(S)$ is $\pgl_2(\dc)\times\pgl_2(\dc)$, hence the kernel of the restriction of $\delta$ to $G$ is a subgroup of $\pgl_2(\dc)\times\pgl_2(\dc)$. Applying Lemma \ref{goursat}, it is of the form $H\times_DH'$, where $H$ and $H'$ are subgroups of $\pgl_2(\dc)$. For the remaining of this proof, we will denote elements of $\aut(S)\cong(\pgl_2(\dc)\times\pgl_2(\dc))\rtimes\dz_2$ as triples $(h,h',a)$, where $h,h'\in\pgl_2(\dc)$, and $a\in\dz_2$. Let $(h,h',0)\in\ker\delta_G$, and let $g=(a,a',1)\in G$. Then $g(h,h',0)g^{-1}=(ah'a^{-1},a'ha'^{-1},0)$, so that, $\ker\delta_G$ being normal in $G$, $aH'a^{-1}=H$. Let $\alpha=(a,I,0)$, and denote $\phi\colon H\rightarrow D$, $\psi\colon H'\rightarrow D$ the morphisms of the fibre product. We have $\alpha(H\times_{D,\phi,\psi}H')\alpha^{-1}=H\times_{D,\phi,\xi}H$, where
    \begin{align}
        \xi\colon H&\rightarrow D\nonumber\\
        h&\mapsto\psi(a^{-1}ha),\nonumber
    \end{align}
    and we have $\ker(\delta_{\alpha G\alpha^{-1}})=(H\times_{D,\phi,\xi}H)\times\{e\}$.
\end{proof}

Let $G$ be a finite subgroup of $\aut(S)$ such that $\rho^G(S)=1$, and $H$ as in Proposition \ref{sgpsautp1p1}.

\begin{lemma}\label{actp1}
    If $G$ is not toric, then $H$ is isomorphic to $\aq$, $\sq$, or $\ac$.
\end{lemma}
\begin{proof}
    Assume that $H$ is isomorphic to $\dz_n$ or $\dih_n$. Then there are points $P_1$ and $P_2$ of $\pl$ which are either fixed by $H$ or forming an orbit of length $2$. Consider the divisors $L_1=\{(P_1,y),y\in\pl\},$ and $L_2=\{(P_2,y),y\in\pl\}$. The divisor $L_1+L_2$ is invariant by $H\times_D H$. Let $g\in G$ such that $G$ is generated by $H\times_DH$ and $g$. This elements swaps the rulings of $\pl\times\pl$, so that the divisor $D=L_1+L_2+g(L_1)+g(L_2)$ forms a square. Let $\alpha\in G$. Since $H\times_DH$ is normal in $G$ and $g^2\in H\times_D H$, we can write $\alpha=hg^i=g^ih'$, for some $h,h'\in H\times_DH$ and $i\in\{0,1\}$. We get:
    \begin{align}
        \alpha(D)&=g^ih(L_1+L_2)+g^ihg(L_1+L_2)\nonumber\\
        &=g^i(L_1)+g^i(L_2)+g^{i+1}h'(L_1+L_2)\nonumber\\
        &=g^i(L_1)+g^i(L_2)+g^{i+1}(L_1)+g^{i+1}(L_2)\nonumber\\
        &=D.\nonumber
    \end{align}
    Finally, groups fixing such a square are conjugated in $\aut(S)$ to a subgroup of $N(\dt)$. Indeed, there is an element $\xi\in\pgl_2(\dc)$ sending $P_1$ to $(1:0)$ and $P_2$ to $(0:1)$. Let $g=(\xi,\xi,0)\in(\pgl_2(\dc)\times\pgl_2(\dc))\times\dz_2\cong\aut(S)$. The image of $D$ by $g$ is the square $C$ as in the proof of Lemma \ref{es1}. Moreover, $gGg^{-1}$ fixes $C$, so that $gGg^{-1}$ is in $N(\dt)$, as implied by Remark \ref{postes1}.
\end{proof}

\begin{proposition}\label{nolinks}
    If $G$ is not toric, then there are no $G$-orbits of length $l\in\{1,2,3,5\}$.
\end{proposition}
\begin{proof}
    Let $P_1,P_2\in\pl$, and $k_1,k_2$ be the respective lengths of their $H$-orbits in $\pl$. The length of the orbit of $(P_1,P_2)$ in $S$ under the action of $H\times_DH$ must be a common multiple $l$ of $k_1$ and $k_2$, and the length of its $G$-orbit is either $l$ or $2l$. Hence, by Lemma \ref{actp1}, it is enough to show that there is no $\aq$-orbit of length $l\in\{1,2,3,5\}$ in $\pl$.
    
    Assume that $H\simeq\aq$. Recall that $\aq$ is unique up to conjugation in $\pgl_2(\dc)$. A simple way\footnote{One could also generate $\aq$ explicitly in $\pgl_2(\dc)$, for example with the matrices $\begin{pmatrix}1&0\\0&\mu_3\end{pmatrix}$ and $\begin{pmatrix}1&2\\1&-1\end{pmatrix}$, where $\mu_3$ is a primitive cube root of the unity.}to see what the $\aq$-orbits in $\pl$ are is to use the action of $\aq$ in $\pgl_3(\dc)$ generated by the matrices $\begin{pmatrix}0&0&1\\1&0&0\\0&1&0\end{pmatrix}$, and $\begin{pmatrix}-1&0&0\\0&1&0\\0&0&1\end{pmatrix}$. It preserves the conic $x^2+y^2+z^2=0$, hence acts on $\pl$ faithfully. We see that there is no fixed point, and there is no orbit of length $2$ either since $\aq$ does not have any subgroup of index $2$. The only subgroup of index $3$ is $\dz_2^2$, generated by the matrices $\begin{pmatrix}-1&0&0\\0&1&0\\0&0&1\end{pmatrix}$ and $\begin{pmatrix}1&0&0\\0&-1&0\\0&0&1\end{pmatrix}$. But it has no fixed point on the conic, hence there is no $H$-orbit of length $3$. Just considering the order of $\aq$, there is no $H$-orbit of length $5$.

    If $H$ is isomorphic to $\sq$ or $\ac$, and since $\aq,\sq,$ and $\ac$ are unique in $\pgl_2(\dc)$ up to conjugation, the lengths $H$-orbits in $\pl$ must be multiples of the lengths of the $\aq$-orbits, namely a multiple of $k\notin\{1,2,3,5\}$. This implies that there is no $H$-orbit of length $k\in\{1,2,3,5\}$ in $\pl$. Summing up what we mentioned in this proof, a $G$-orbit can only be a multiple of $k$, for some positive integer $k\notin\{1,2,3,5\}$.
\end{proof}

\begin{remark}
    There is a non-toric subgroup of $\aut(S)$ isomorphic to $\aq\rtimes\dz_2$, which has a $G$-orbit of length four in general position. Recall from Remark \ref{glinksp1p1}, that there is a $G$-link centered in such orbit. It is of the following form.
    
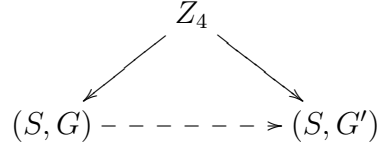
\[\xymatrix
    {
         &Z_4\ar[ld]\ar[rd]&\\
         (S,G)\ar@{-->}[rr]&&(S,G')
    }
    \]
    By Proposition \ref{nolinks}, the only way for $S$ not to be $G$-solid if $G$ is a toric group with $\rho^G(S)=1$ would be to get the above diagram with $G'$ non-toric. To exclude this possibility, one could use a result of \cite{egor} stating that such link leads must lead to a $G$-isomorphic surface. We will use another method based on general group theory. It will have the advantage of giving a simple characterization of toric subgroups of $\aut(S)$, making the final result much easier to read. It turns out that a finite toric subgroup $G$ of $\aut(S)$ with $\rho^G=1$ cannot be isomorphic to a non-toric subgroup of $\aut(S)$ with $\rho^G(S)=1$.
\end{remark}

\begin{proposition}\label{aqnottor}
    Let $G$ be a subgroup of $\aut(S)$ such that $\rho^G(S)=1$. Then $G$ is toric if and only if is has no subgroup isomorphic to $\aq$.
\end{proposition}

We will step back to some general group theory to achieve the above result. To prove that a non-toric group contains $\aq$, it is enough to show that $H\times_DH$ contains a subgroup isomorphic to $H$. The fact that a fibre product of this form contains a subgroup isomorphic to $H$ is not always true if we consider any group $H$, even though working on many cases may suggest it is. In the following example, the group $H$ has the smallest possible order such that there are surjective morphisms $\phi,\psi\colon H\rightarrow D$ making the $H\times_DH$ a group which does not contain any subgroup isomorphic to $H$.

\begin{example}\label{contrex}
    Let $H\simeq \dih_4=\left<a,b|a^4=b^2=(ab)^2=1\right>$, take $D=\dz_2^2$, and
    \begin{align}
        \phi\colon a&\mapsto(0,1)\nonumber\\
        b&\mapsto(1,0)\nonumber\\
        \psi\colon a&\mapsto(1,0)\nonumber\\
        b&\mapsto(1,1).\nonumber
    \end{align}
    Then $H\times_DH$ is isomorphic to $\dz_2^2\rtimes\dz_4$, and does not have any subgroup isomorphic to $H$. This example exists in our context. Let $H=\left<a,b\right>$, where $a=\begin{pmatrix}
            1&0\\0&i
        \end{pmatrix}$, and $b=\begin{pmatrix}
            0&1\\1&0
        \end{pmatrix}$, and define the morphisms $\phi$ and $\psi$ as above. To have a subgroup $G$ of $\aut(S)$ such that $\rho^G(S)=1$ and $G\cap(\pgl_2(\dc)\times\pgl_2(\dc))=H\times_{\dz_2^2}H$, we take $G=\left<H\times_{\dz_2^2}H,g\right>$, where $g=(\mu_8^{-1}y,\mu_8x)$, with $\mu_8$ a square root of $i$.
\end{example}

However, it is true that if a group $H$ splits in a nice way, then a fibre product of the form $H\times_DH$ will have a subgroup isomorphic to $H$. It will apply in particular for the groups we need, namely $\aq,\sq$, and $\ac$.

\begin{lemma}\label{sgpfibre}
    Let $H$ be a group of the form $N\rtimes D$, and $\phi,\psi\colon H\rightarrow D$ be surjective group morphisms such that $\ker(\phi)=\ker(\psi)=N$. Then $H\times_DH$ has a subgroup isomorphic to $H$.
\end{lemma}
\begin{proof}
    The subgroups $N=\ker(\psi)$ and $D$ having trivial intersection in $H$ means that the restriction of $\psi$ to $D$ is an isomorphism. So we can define the set $\tilde D=\{(d,\psi^{-1}\circ\phi(d)),d\in D\}$. It is a subgroup of $H\times_DH$, isomorphic to $D$. Denote $\tilde N$ the subgroup $\ker(\phi)\times\{\id\}$ of $H\times H$. It is a subgroup of $H\times_DH$, isomorphic to $\ker(\phi)=N$. Consider the projection $\pi$ from $H\times H$ to the first factor $H$. Its restriction to $\tilde N\tilde D$ is surjective by construction of $\tilde N$ and $\tilde D$. If $\pi(g,g')=e$, then $g=e$. But since $\tilde N$ and $\tilde D$ have trivial intersection, $e=e.e$ is the only decomposition of $e$ into a product of an element of $N$ and an element of $D$. Hence, we get that $(g,g')=(e.e,e.\psi^{-1}\circ\phi(e))=(e.e)$, so that $\pi$ restricts to an isomorphism between $\tilde N\tilde D\subset H\times_DH$ and $H$.
\end{proof}

We get the following immediate consequence.

\begin{corollary}
    In the following cases, a fibre product of the form $H\times_DH$ contains a subgroup isomorphic to $H$.
    \begin{itemize}
        \item $H$ is simple,
        \item $H$ is cyclic,
        \item $H$ is isomorphic to $\mathfrak A_n$ or $\mathfrak S_n$ for some $n$.
    \end{itemize}
\end{corollary}

We can now prove Proposition \ref{aqnottor}.
\begin{proof}[Proof of \ref{aqnottor}]
    Recall that the normalizer $N(\dt)$ of the torus $\dt$ in $\aut(S)$ satisfies this exact sequence:
    \[\begin{tikzcd}
	1 & {\dt} & {N(\dt)} & {\dih_4} & {1.}
	\arrow[from=1-1, to=1-2]
	\arrow[from=1-2, to=1-3]
	\arrow["w", from=1-3, to=1-4]
	\arrow[from=1-4, to=1-5]
\end{tikzcd}\]
But there is no subgroup of $\dih_4$ isomorphic to a quotient of $\aq$ by an abelian group. Hence, if $G$ is toric, it cannot have a subgroup isomorphic to $\aq$. Conversely, assume that $G$ is a finite toric subgroup of $\aut(S)$ with $\rho^G(S)=1$. By Proposition \ref{sgpsautp1p1}, it satisfies the exact sequence
\[\begin{tikzcd}
    	1 & {H\times_DH} & {G} & {\dz_2} & {1,}
    	\arrow[from=1-1, to=1-2]
    	\arrow[from=1-2, to=1-3]
    	\arrow[from=1-3, to=1-4]
    	\arrow[from=1-4, to=1-5]
\end{tikzcd}\] with $H$ isomorphic to $\aq$, $\sq$, or $\ac$. By Lemma \ref{sgpfibre}, the group $G$ has a subgroup isomorphic to $H$. Hence, it has a subgroup isomorphic to $\aq$.
\end{proof}

Summing up the above results, we get the following.
\begin{proposition}
    Let $G$ be a finite non-toric subgroup of $\aut(S)$. Then $S$ is $G$-solid.
\end{proposition}

Finally, here is the classification of subgroups $G$ of $\aut(S)$ such that $S$ is $G$-solid.

\begin{theorem}\label{gsolp1p1}
    Let $G$ be a finite subgroup of $\aut(S)$, such that $\rho^G(S)=1$. Then $S$ is not $G$-solid if and only if $G$ is toric, and in one of the following cases.
    \begin{itemize}
        \item The group $G$ is not conjugated in $\aut(S)$ to a group containing $r=(\frac{1}{y},x)$.
        \item $G$ is conjugated in $\aut(\pl\times\pl)$ to one of the following groups.
        \begin{itemize}
            \item $\dz_4$, generated by $r=(\frac{1}{y},x)$,
            \item $\dih_4$, generated by $r$, and $(y,x)$,
            \item $\dih_4$, generated by $r$ and $(-y,-x)$,
            \item $\dz_4\times\dz_2$, generated by $r$ and $t=(-x,-y)$,
            \item $\dz_2\times \dih_4$, generated by $t$ and $r$ as above, and $h=(y,x)$.
        \end{itemize}
    \end{itemize}
\end{theorem}

Once again, Proposition \ref{aqnottor} not only allows us to conclude about the $G$-solidity of the non-toric finite sugroups of $\aut(S)$, but also gives a way to reformulate Theorem \ref{gsolp1p1} without mentioning the toric structure of $S$, giving an equivalent statement which is shorter and easier to read. Let $G_{16}\cong\dz_2\times \dih_4$ be the subgroup of $\aut(S)$ generated by $r=(\frac{1}{y},x),s=(y,x),$ and $t=(-x,-y)$.
\begin{theorem}
    Let $G$ be a finite subgroup of $\aut(S)$, such that $\rho^G(S)=1$. Then $S$ is $G$-solid if and only if, up to conjugation in $\aut(S)$,
    \begin{itemize}
        \item either $\aq\subset G$,
        \item or $(r\in G$ and $G\not\subset G_{16})$.
    \end{itemize}
\end{theorem}

\section{Del Pezzo surfaces of degree $6$}\label{dp6}

Up to isomorphism, there is only one smooth del Pezzo surface of degree $6$. It is obtained by blowing-up $\mathbb P^2$ in three points $P_1,P_2$ and $P_3$ in general position. We will denote by $E_i$ the exceptional curve contracted to the point $P_i$, and by $\dih_{ij}$ the proper transform of the line passing through $P_i$ and $P_j$. Recall that there is a split exact sequence \[\begin{tikzcd}
    	1 & {\dt} & {\aut(S)} & {\dih_6} & {1,}
    	\arrow[from=1-1, to=1-2]
    	\arrow[from=1-2, to=1-3]
    	\arrow["w", from=1-3, to=1-4]
    	\arrow[from=1-4, to=1-5]
    \end{tikzcd}\] where $w$ is given by the action of $\aut(S)$ on the hexagon formed by the $(-1)$-curves of $S$. The group $\dt\cong(\dc^*)^2$ is the lift in $\aut(S)$ of the diagonal automorphisms of $\pp$, which fixes $P_1$, $P_2$, and $P_3$. This subgroup will be denoted by $\dt$.

We will use the embedding of $S$ in $\pp\times\pp$ given by $xu=yv=zw$, where $(x:y:z)\times(u:v:w)$ stands for the coordinates in $\pp\times\pp$. This model is presented more extensively in \cite{theseblanc}. Explicitly, an element $(a,b)\in(\dc^*)^2$ corresponds to the map $\dps\rightarrow(x:ay:bz)\times(u:a^{-1}v:b^{-1}w)$. The maps
\begin{align}
    r&\colon \dps\mapsto(w:u:v)\times(z:x:y),\text{ and}\nonumber\\
    s&\colon \dps\mapsto(x:z:y)\times(u:w:v)\nonumber
\end{align}
generate a subgroup of $\aut(S)$ isomorphic to $\dih_6$, and the quotient $\aut(S)/\dt\cong \dih_6$ is generated by the images of $r$ and $s$. The automorphism $r$ acts on the hexagon formed by the $(-1)$-curves as an elementary rotation, and $s$ acts as a reflection of the hexagon which does not fix any vertex. We have the relations $r^6=s^2=(rs)^2=\mathrm{id}$, giving the classical presentation of $\dih_6$.

\begin{lemma}\label{conjr}
    Any element $r'\in\aut(S)$ such that $w(r')=w(r)$ is equal to $r$ up to conjugation by an element of $\dt$.
\end{lemma}
\begin{proof}
    Let $r'\in\aut(S)$ be such that $w(r')=w(r)$. Since the kernel of $w$ is the normal subgroup $\dt$ of $\aut(S)$, we have $r'=tr,$ for some $t\in\dt$. Explicitly, there exist $(a,b)\in(\dc^*)^2$ such that $r'\colon\dps\mapsto(w:au:bv)\times(z:a^{-1}x:b^{-1}y)$. Let $t\colon\dps\mapsto(x:cy:dz)\times(u:c^{-1}v:d^{-1}w)\in\dt\cong(\dc^*)^2$. We have $$tr't^{-1}\colon\dps\mapsto(w:acd^{-1}u:bcv)\times(z:a^{-1}c^{-1}dx:b^{-1}c^{-1}y).$$ Setting $c=b^{-1}$ and $d=ab^{-1}$, we get $tr't^{-1}=r$.
\end{proof}

Let $G$ be a subgroup of $\aut(S)$, such that $\rho^G(S)=1$. The classification of Sarksov links in \cite{iskovskikh1996factorization} and the fact that Bertini and Geiser involutions lead to $G$-isomorphic surfaces imply the following.

\begin{remark}\label{fix}
    There is no link of type \MakeUppercase{\romannumeral 1} starting from $S$. Hence, for $S$ not to be $G$-solid, it has to be $G$-birational to a surface $S'$, not isomorphic to $S$. The only $G$-link $S\dashrightarrow S'$ such that $S'$ is not isomorphic to $S$ is the blow-up of a point $P\in S$ which is not in the exceptional locus, followed by the contraction of three $(-1)$-curves. We obtain $S\cong\pl\times\pl$. Hence, any $G$-birational map from $S$ to a $G$-conic bundle $S''\rightarrow\pl$ must split in the following way:
\[\begin{tikzcd}
	S & \pl\times\pl & {S''.}
	\arrow[dashed, from=1-1, to=1-2]
	\arrow[dashed, from=1-2, to=1-3]
\end{tikzcd}\] 
We will often refer to the results of section \ref{dp8} to determine whether or not the surface $S$ is $G$-solid.
\end{remark}

In particular, Remark \ref{fix} implies the following lemma:

\begin{lemma}\label{notind6}
    Let $G$ be a subgroup of $\aut(S)$ such that $\rho^G(S)=1$, not isomorphic to a subgroup of $\dih_6$. Then $S$ is $G$-solid.
\end{lemma}
\begin{proof}
    Assume that $S$ is not $G$-solid. Then, by Remark \ref{fix}, there exists a subgroup $G'$ of $\aut(S)$ birationally conjugated to $G$, which fixes a point $P$ in general position. But in this case, $G'\cap\dt=\mathrm{id}$, so that $G'$ is mapped isomorphically by $w$ to a subgroup of $\dih_6$.
\end{proof}

\begin{lemma}\label{necconddp6}
    If $S$ is a $G$-del Pezzo surface, then the image of $G$ by $w$ in $\dih_6$ must contain the subgroup of $\dih_6$ isomorphic to $\dz_6$, or the subgroup isomorphic to $\st$ acting transitively on the $(-1)$-curves of $S$.
\end{lemma}
\begin{proof}
    Assume that $G$ does not act transitively on the $(-1)$-curves of $S$. Just straightforwardly checking all the possible subgroups of $\dih_6$, one can check that one of the divisors $E_1+E_2+E_3$, or $\dih_{12}+\dih_{23}+\dih_{13}$, or $E_i+E_{jk}$, with $j,k\ne i$, is invariant by $G$. In all those cases, by Castelnuovo's contractibility criterion, there exists a $G$-birational morphism either to $\pp$ or to $\pl\times\pl$.
\end{proof}

\begin{corollary}
    Let $G$ be a subgroup of $\aut(S)$ such that $\rho^G(S)=1$. Then $G$ is not isomorphic to any of the groups $\dz_2^2,\dz_3,\dz_2,\{\id\}$.
\end{corollary}

The only remaining groups of interest are $\dih_6$, $\dz_6$, and $\st$. Let us start with the case of $\st$.

\begin{proposition}\label{s3}
	Let $G$ be a subgroup of $\aut(S)$ isomorphic to $\st$ and such that $\rho^G(S)=1$. Then, up to conjugation in $\aut(S)$, the group $G$ is generated by $\dps\mapsto(y:z:x)\times(v:w:u)$ and $\dps\mapsto(w:v:u)\times(z:y:x)$. Moreover, the surface $S$ is not $G$-solid.
\end{proposition}

\begin{proof}
    If the restriction of $w$ to $G$ is not injective, then the image of $G$ by $w$ is either isomorphic to $\dz_3$, isomorphic to $\dz_2$, or trivial. By Lemma \ref{necconddp6}, it implies that $\rho^G(S)=1$, contradicting our assumption. Hence $w(G)$ is isomorphic to $\st$ and acts transitively on the $(-1)$-curves, by Lemma \ref{necconddp6}. We deduce that the group $G$ is generated by an element $g$ such that $w(g)=w(r^2)$, and an element $h$ such that $w(h)=w(rs)$. Geometrically, $rs$ acts on the hexagon formed by the $(-1)$-curves as a reflection which fixes two opposite vertices. By Lemma \ref{conjr}, we have $g\colon\dps\mapsto(y:z:x)\times(v:w:u)$ up to conjugation by an element of the torus. Since $\ker(w)=\dt$, $h=trs$ for some $t\in\dt$. Explicitly, $h$ is of the form $\dps\mapsto(w:av:bu)\times(z:a^{-1}y:b^{-1}x)$, for some $(a,b)\in(\dc^*)^2$. We get $h^2\colon\dps\mapsto(b^{-1}x:y:bz)\times(bu:v:b^{-1}w)$, and knowing that $\ord(h)=2$, we deduce that $b=1$. Moreover, $hgh\colon\dps\mapsto(a^{-1}y:az:x)\times(av:a^{-1}w:u)$, but relations in $\dih_6$ imply that $hgh=g^{-1}\colon\dps\mapsto(y:z:x)\times(v:w:u)$. Hence, $a=1$, so that $h\colon\dps\mapsto(w:v:u)\times(z:y:x)$. The three points of the form $(1:\mu^{2k}:\mu^k)\times(1:\mu^k:\mu^{2k})$ are in general position and fixed by the action of $G$. Blowing up one of them, we get a $G$-link to the surface $\pl\times\pl$ with two fixed points on it in general position. Hence, we are in the case of Example \ref{2fixed}, so that $S$ is $G$-birational to the $G$-conic bundle $\mathbb F_1$.
\end{proof}

\begin{proposition}\label{d6}
    If $G$ is a subgroup of $\aut(S)$ isomorphic to $\dih_6$ such that $\rho^G(S)=1$, then, up to conjugation in $\aut(S)$, the group $G$ is generated by $\dps\mapsto(w:u:v)\times(z:x:y)$ and $\dps\mapsto(x:z:y)\times(u:w:v)$. Moreover, the surface $S$ is not $G$-solid.
\end{proposition}
\begin{proof}
    Assume that $G\cong \dih_6$. Going through the possible quotients of $\dih_6$ and combining with Lemma \ref{necconddp6}, we see that either $w(G)=\dih_6$, or $w(G)=\st$. Let us exclude the latter case, in which $G\cap\ker(w)=T\cong\dz_2$. Since the only extension of the form
    \[\begin{tikzcd}
    	1 & {\dz_2} & {G} & {\st} & {1}
    	\arrow[from=1-1, to=1-2]
    	\arrow[from=1-2, to=1-3]
    	\arrow[from=1-3, to=1-4]
    	\arrow[from=1-4, to=1-5]
    \end{tikzcd}\]
    splits, and since $\dz_2$ acts trivially on the Picard group of $S$, we deduce that the group $G$ has a subgroup $H$ isomorphic to $\st$ such that $\pic^H(S)=1$. But such group is given explicitly in Lemma \ref{s3}, and we see that it cannot commute with a subgroup of $\dt$ isomorphic to $\dz_2$. Hence we get $G\cong w(G)=\dih_6$, and the group $G$ is generated by an element $g$ such that $w(g)=w(r)$, and an element $h$ such that $w(h)=w(s)$. By Lemma \ref{conjr}, the automorphism $g$ it conjugated to $r$ by an element of $\dt$, and its unique fixed point is $P=(1:1:1)\times(1:1:1)$. The isomorphism $h$ is of the form $\dps\mapsto(x:az:bz)\times(u:a^{-1}w:b^{-1}v)$. Since $\ord(h)=2$, be get $b=a^{-1}$. Moreover, $hr^2h\colon\dps\mapsto(az:x:a^{-2}y)\times(a^{-1}w:u:a^2v)$. But the structure of $\dih_6$ implies that $hr^2h=r^{-2}\colon\dps\mapsto(z:x:y)\times(w:u:v)$, so that $a=1$. Hence, $h=s\colon\dps\mapsto(x:z:y)\times(u:w:v)$. This automorphism also fixes $P$, so that, as described in Remark \ref{fix}, there exists a $G$-link from $S$ to $\pl\times\pl$ centered at $P$. By Theorem \ref{gsolp1p1}, the surface $\pl\times\pl$ is not $G'$-solid for any subgroup $G'\subset\aut(\pl\times\pl)$ isomorphic to $\dih_6$. We conclude that $S$ is not $G$-solid.
\end{proof}

\begin{proposition}\label{hsol}
    Let $G$ be a subgroup of $\aut(S)$ isomorphic to $\dz_6$ and such that $\rho^G(S)=1$. Then, up to conjugation in $\aut(S)$, the group $G$ is generated by $\dps\mapsto(w:u:v)\times(z:x:y)$. Moreover,the surface $S$ is not $G$-solid.
\end{proposition}
\begin{proof}
    First, notice that the restriction of $w$ to $G$ is injective. Indeed, if not, then the image of $G$ by $w$ is either isomorphic to $\dz_3$, isomorphic to $\dz_2$, or trivial. By Lemma \ref{necconddp6}, it implies that $\rho^G(S)>1$, contradicting our assumption. The group $G$ is then generated by an element $g$ such that $w(g)=w(r)$. By Lemma \ref{conjr}, the automorphism $g$ is conjugated to up to $r\colon\dps\mapsto(w:u:v)\times(z:x:y)$ by an element of $\dt$. The only fixed point of $r$ on $S$ is $(1:1:1)\times(1:1:1)$. In particular, there is a subgroup $G'$ of $\aut(S)$ containing $G$ and isomorphic to $\dih_6$. Since $S$ is not $G'$-solid by Proposition \ref{d6}, the surface $S$ is not $G$-solid either.
\end{proof}



Summing up the results of this section, we get the following.

\begin{theorem}
    Let $G$ be a subgroup of $\aut(S)$ such that $\rho^G(S)=1$. Then $S$ is $G$-solid if and only if $G$ is not isomorphic to $\dz_6$, $\st$, or $\dih_6$.
\end{theorem}

\section{The projective plane}

The only remaining smooth del Pezzo surface is $S=\pp$, whose automorphism group is $\pgl_3(\dc)$. The $G$-rigidity of $S$ has been studied D.Sakovics in \cite{sako}. We will point out how his results hold for the $G$-solidity of $S$, and describe the full $G$-birational geometry of $S$ for $G\subset\pgl_3(\dc)$ isomorphic to $\aq$ or $\sq$. In other words, we are going to list all the $G$-Mori fibre spaces $S'$ such that there exists a $G$-birational map $S\dashrightarrow S'$.

\begin{theorem}[\cite{sako}]
    The projective plane is $G$-rigid if and only if $G$ is transitive and not isomorphic to $\sq$ or $\aq$.
\end{theorem}

Moreover, if $G$ fixes a point on $S$, then we can $G$-equivariently blow-up this point and get a $G$-birational map to the $G$-conic bundle $\mathbb F_1$, so that $S$ is not $G$-solid. Hence, the only remaining cases to study are those of $\sq$ and $\aq$.

\begin{lemma}\label{conjaqsq}
    The subgroups of $\pgl_3(\dc)$ isomorphic to $\sq$ or $\aq$ are unique up to conjugation in $\pgl_3(\dc)$.
\end{lemma}
\begin{proof}
    The canonical projection $\pi:\gl_3(\mathbb C)\rightarrow\mathrm{PGL}_3(\mathbb C)$ induces a surjection $\mathrm{SL}_3(\mathbb C)\rightarrow\pgl_3(\dc)$. Let $G\subset\pgl_3(\dc)$ be a subgroup of $\pgl_3(\dc)$ isomorphic to $\mathfrak S_4$, and consider its lift $G'\subset\mathrm{SL}_3(\mathbb C)$ by the above projection. The kernel of $\pi$ restricted to $\specl_3(\dc)$ is $\{I_3,\mu I_3,\mu^2 I_3\}$, where $\mu$ is a primitive cube root of the unity. Thus, we have an extension
    \[\begin{tikzcd}
    	1 & {\dz_3} & {G'} & {G} & {1.}
    	\arrow[from=1-1, to=1-2]
    	\arrow[from=1-2, to=1-3]
    	\arrow[from=1-3, to=1-4]
    	\arrow[from=1-4, to=1-5]
    \end{tikzcd}\]
   But $\{I_3,\mu I_3,\mu^2 I_3\}$ lies in the center of $\gl_3(\mathbb C)$. We deduce that $G'$ is isomorphic to $\dz_3\times\sq$, since this group is the only triple central extension of $\sq$. Its subgroup $\{\id\}\times\sq$ is sent isomorphically to $G$ by $\pi$. In particular, there exists a subgroup of $\gl_3(\mathbb C)$ isomorphic to $\mathfrak S_4$, whose projection in $\pgl_3(\dc)$ is $G$. But the only irreducible faithful linear representations of degree $3$ of $\mathfrak S_4$, up to equivalence of representations, are the standard one and its product with the sign representation. Both are mapped by $\pi$ to the same subgroup of $\pgl_3(\dc)$.
   
The group $\aq$ has two triple central extensions, namely $\dz_3\times\aq$, and a non-split extension. But in the second case, there is no irreducible faithful representation of degree $3$ whose image is in $\specl_3(\dc)$. Hence, as in the case of $\sq$, there is a subgroup of $\gl_3(\dc)$ mapped isomorphically by $\pi$ onto $G$. Since there is only one equivalence class of irrecucible linear representations of $\aq$ of degree $3$, we conclude that $\aq$ is unique in $\pgl_3(\dc)$, up to conjugation.
\end{proof}

\begin{proposition}\label{sqmf}
    Let $G\cong\sq$ be a subgroup of $\aut(S)$. The only $G$-links starting from $S$ are:
    \begin{itemize}
        \item A link of type \MakeUppercase{\romannumeral 1} of the form
        $$
        \xymatrix
            {
                 &Z_5\ar[ld]_\sigma\ar[d]^\pi\\
                 S&\pl
            }\label{linkounou}
        $$
        where $\pi:Z_5\rightarrow\pl$ is a $G$-conic bundle on a del Pezzo surface of degree $5$.
        \item A link of type \MakeUppercase{\romannumeral 2} of the form
        $$
        \xymatrix
            {
                 &Z_6\ar[ld]_\sigma\ar[rd]^\tau&\\
                 S\ar@{-->}[rr]^\tau&&S
            }
        $$
        where $Z_6$ is a del Pezzo surface of degree $6$, and $\tau$ is the standard Cremona involution.
    \end{itemize}
    Moreover, the only $G$-link starting from $Z_5$ is the inverse of \eqref{linkounou}, leading back to $S$.
\end{proposition}

\begin{proof}
    Let $G$ be the subgroup of $\pgl_3(\dc)$ generated by $A=\begin{pmatrix}-1&0&0\\0&1&0\\0&0&1\end{pmatrix}$, $B=\begin{pmatrix}0&0&1\\1&0&0\\0&1&0\end{pmatrix}$, and $C=\begin{pmatrix}1&0&0\\0&0&1\\0&1&0\end{pmatrix}$. This group is isomorphic to $\sq$. Recall that two isomorphic subgroups of $\sq$ are always conjugated in $\sq$, so that it is enough to find an occurrence of each subgroup up to isomorphism in the study of the possible stabilizers.
    \begin{itemize}
        \item There is no fixed point under the above action.
        \item Notice that the subgroup of $G$ generated by $A$ and $B$ is isomorphic to $\mathfrak A_4$. It is the only subgroup of $G$ of index two and has no fixed point, so that $G$ does not have any orbit of length $2$.
        \item The only subgroup of index three up to conjugation is $\dih_4$, generated by $\begin{pmatrix}-1&0&0\\0&0&-1\\0&1&0\end{pmatrix}$, and $\begin{pmatrix}1&0&0\\0&0&1\\0&1&0\end{pmatrix}$. Its fixed points are $O_3:=\{(1:0:0),(0:1:0),(0:0:1)\}$, and $O_3$ is the only orbit of length $3$ under the action of $G$. It is the center of the link of type \MakeUppercase{\romannumeral 2}:
        $$
        \xymatrix
            {
                 &Z_6\ar[ld]_\sigma\ar[rd]^\tau&\\
                 S\ar@{-->}[rr]^\tau&&S
            }
        $$
        where $Z_6$ is a del Pezzo surface of degree $6$, and $\tau$ is the standard Cremona involution.
        \item The group $G$ does have any orbit of length $4$ in general position. Indeed, $G$ has a unique subgroup of index $4$. It is isomorphic to $\st$, and generated by $\begin{pmatrix}0&0&1\\1&0&0\\0&1&0\end{pmatrix}$ and $\begin{pmatrix}1&0&0\\0&0&1\\0&1&0\end{pmatrix}$. The only point fixed by $\st$ is $(1:1:1)$, and its $G$-orbit is $O_4:=\{(1:1:1),(-1:1:1),(1:-1:1),(1:1:-1)\}$. Hence, there is a $G$-link of type \MakeUppercase{\romannumeral 1} of the form:
        \begin{align}
        \xymatrix
            {
                 &\mathbb F_1\ar[ld]_\sigma\ar[d]^\pi\\
                 S&\pl
            }\label{link1}
        \end{align}
        where $\pi:X\rightarrow\pl$ is a $G$-conic bundle on a del Pezzo surface of degree $5$, with invariant Picard rank $2$.
        \item There is an orbit of length $6$, but not in general position. Indeed, there are two subgroups of index $6$ in $G$. The first one is $\dz_2^2$, generated by $A=\begin{pmatrix}-1&0&0\\0&1&0\\0&0&1\end{pmatrix}$, and $B=\begin{pmatrix}1&0&0\\0&-1&0\\0&0&1\end{pmatrix}$. Its only fixed point is $(0:0:1)$, and its orbit is $O_3$, which is of length $3$. The other subgroup of index $6$ of $G$ is $Z_4$, generated by $\begin{pmatrix}1&0&0\\0&0&1\\0&-1&0\end{pmatrix}$. Its fixed points are $(0:-i:1)$, $(0:i:1)$, and $(1:0:0)$. The orbit of the last one is $O_3$, and the two others have the orbit $O_6:=\{(0:i:1),(0:-i:1),(1:0:i),(1:0:-i),(i:1:0),(-i:1:0)$. These six points are not in general position, as they all lie on the Fermat conic $x^2+y^2+z^2=0$.
        \item There is no orbit of length $8$ in general position. The only subgroup of index $8$ is $\dz_3$, generated by $\begin{pmatrix}0&0&1\\1&0&0\\0&1&0\end{pmatrix}$. Its fixed points are $(1:1:1)$, $(1:\mu_3:\mu_3^2)$, and $(1:\mu_3^2:\mu_3)$, where $\mu_3$ is a primitive cube root of the unity. The orbit of $(1:1:1)$ is of length $4$, and $(1:\mu_3:\mu_3^2)$, and $(1:\mu_3^2:\mu_3)$ lie on the same orbit of length $8$: $O_8:=\{(1:\mu_3:\mu_3^2),(1:\mu_3:\mu_3^2),(1:\mu_3:\mu_3^2),(1:\mu_3:\mu_3^2),(1:\mu_3^2:\mu_3),(1:\mu_3^2:\mu_3),(1:\mu_3^2:\mu_3),(1:\mu_3^2:\mu_3)\}$. But these eight points are on the conic $x^2=yz$.
    \end{itemize}
    It remains to show that there is no $G$-link starting from $X$, except the inverse of the link \eqref{link1}. For this, we will show that all the orbits under the action of $G$ lifted on $\mathbb F_1$ have several points on the same fibre of the conic bundle. The birational map $\gamma$ is given by the linear system $\vert C\vert$ of conics passing through all the points of the orbit $O_4$. The curves $x^2-y^2=0$ and $x^2-z^2=0$ form a basis of this linear system. Hence, up to a change of basis, the map $\gamma$ is of the form $(x:y:z)\mapsto(x^2-y^2:x^2-z^2)$. The image of $G$ by $\gamma$ is isomorphic to $\st$. Hence the kernel $N$ of the induced morphism $G\rightarrow\aut(\pl)$ is isomorphic to $\dz_2^2$, and generated by the matrices $A=\begin{pmatrix}-1&0&0\\0&1&0\\0&0&1\end{pmatrix}$, and $B=\begin{pmatrix}1&0&0\\0&-1&0\\0&0&1\end{pmatrix}$. The action of $N$ on the smooth conics of this system is faithful, hence $N$ does not fix any point on the regular fibres of the conic bundle.
\end{proof}

We get the following immediate consequences.

\begin{corollary}
    Let $G$ be a subgroup of $\aut(S)$ isomorphic to $\sq$. The projective plane is not $G$-solid, but not $G$-birational to any Hirzebruch surface.
\end{corollary}

\begin{corollary}
    Let $G$ be a subgroup of $\aut(S)$ isomorphic to $\sq$. Then $\mathrm{Bir}^G(S)=\left<G,\tau\right>\cong\sq\times\dz_2$, where $\tau$ is the standard Cremona involution.
\end{corollary}
\begin{proof}
    The elements of $G$ are the automorphisms of the form $(x:y:z)\mapsto\sigma(\alpha x:\beta y:\gamma z)$, where $\sigma$ is a permutation of the coordinates, and $\alpha,\beta,\gamma\in\{-1,1\}$. The involution $\tau:(x:y:z)\dashrightarrow(yz:xz:xy)$ commutes with all these elements.
\end{proof}

The remaining case to study is that of $G\cong\aq$.

\begin{proposition}\label{aqmf}
    Let $G\cong\aq$ be a subgroup of $\aut(S)$. The only $G$-links starting from $S$ are:
    \begin{itemize}
        \item Three links of type \MakeUppercase{\romannumeral 1} of the form
        \begin{align}
        \xymatrix
            {
                 &Z_5\ar[ld]_\sigma\ar[d]^\pi\\
                 S&\pl
            }\label{linkou}
        \end{align}
        
        where $\pi:X\rightarrow\pl$ is a $G$-conic bundle on a del Pezzo surface of degree $5$.
        \item A link of type \MakeUppercase{\romannumeral 2} of the form
        $$
        \xymatrix
            {
                 &Z_6\ar[ld]_\sigma\ar[rd]^\tau&\\
                 S\ar@{-->}[rr]^\tau&&S
            }
        $$
        where $Z_6$ is a del Pezzo surface of degree $6$, and $\tau$ is the standard Cremona involution.
        \item A one parameter family of links of type \MakeUppercase{\romannumeral 2} of the form
        $$
        \xymatrix
            {
                 &Z\ar[ld]_\sigma\ar[rd]^\tau&\\
                 S\ar@{-->}[rr]^{i_a}&&S
            }
        $$
        where $\sigma$ is the blow-up of an orbit of six points, $\tau$ is the $G$-equivariant contraction of eight $(-1)$-curves, and $i_a$ is a birational involution.
    \end{itemize}
    The only $G$-link starting from $X$ is the inverse of \eqref{linkou}, leading back to $S$.
\end{proposition}
\begin{proof}
    Let $G$ be a subgroup of $\pgl_3(\dc)$ isomorphic to $\sq$. Noticing that the matrices $a:=\begin{pmatrix}0&0&1\\1&0&0\\0&1&0\end{pmatrix}$, and $b:=\begin{pmatrix}-1&0&0\\0&1&0\\0&0&1\end{pmatrix}$ generate a subgroup of $\pgl_3(\dc)$ isomorphic to $\aq$, and using Lemma \ref{conjaqsq}, we gte that $G$ is conjugated to $\left<a,b\right>$ in $\pgl_3(\dc)$. Also recall that any two isomorphic subgroups of $\aq$ are conjugated to each other.
    \begin{itemize}
        \item There is no fixed point under the above action.
        \item There is no subgroup of index $2$ in $G$.
        \item The only subgroup of index $3$ is $N=\dz_2^2$, generated by $A=\begin{pmatrix}-1&0&0\\0&1&0\\0&0&1\end{pmatrix}$, and $B=\begin{pmatrix}1&0&0\\0&-1&0\\0&0&1\end{pmatrix}$. Its only fixed point is $(0:0:1)$, whose orbit is $O_3=\{(1:0:0),(0:1:0),(0:0:1)\}$. The only link centered at this orbit is- the link of type \MakeUppercase{\romannumeral 2} of the form
        $$
        
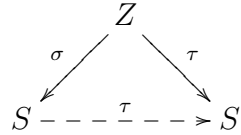
\xymatrix
            {
                 &Z\ar[ld]_\sigma\ar[rd]^\tau&\\
                 S\ar@{-->}[rr]^\tau&&S
            }
        $$
        where $\tau$ is the standard Cremona involution.
        \item The only subgroup of $G$ of index $4$ is $\dz_3$, generated by $\begin{pmatrix}0&0&1\\1&0&0\\0&1&0\end{pmatrix}$, and have three independant fixed points: $(1:1:1)$, $(1:\mu_3:\mu_3^2)$, and $(1:\mu_3^2:\mu_3)$. They give rise to three distinct orbits of length $4$ in general position: $O_4:=\{(1:1:1),(-1:1:1),(1:-1:1),(1:1:-1)\}$, $O'_4:=\{(1:\mu_3:\mu_3^2),(-1:\mu_3:\mu_3^2),(1:-\mu_3:\mu_3^2),(1:\mu_3:-\mu_3^2)\}$, and $O''_4:=\{(1:\mu_3^2:\mu_3),(-1:\mu_3^2:\mu_3),(1:-\mu_3^2:\mu_3),(-1:-\mu_3^2:\mu_3)\}$. In each case, the points are in general position, and blowing-up one of them, we get a $G$-link of type \MakeUppercase{\romannumeral 1} of the form
        
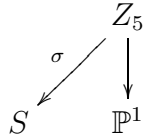
\begin{align}
        \xymatrix
            {
                 &Z_5\ar[ld]_\sigma\ar[d]\\
                 S&\pl
            }\label{link2}
        \end{align}
        where $Z_5$ is a del Pezzo surface of degree $5$.
        \item The unique subgroup of index $6$ of $G$ is isomorphic to $\dz_2$, generated by $\begin{pmatrix}-1&0&0\\0&1&0\\0&0&1\end{pmatrix}$. Its fixed points are the points of $O_3$, and those of the form $(0:1:a)$, with $a\ne0$. They form orbits of length $6$ of the form $O_6^a:=\{(0:1:a),(a:0:1),(1:a:0),(0:-1:a),(a:0:-1),(-1:a:0)\}$. The only Sarkisov link centered at such orbit is the link of type \MakeUppercase{\romannumeral 2} of the form
        $$
        
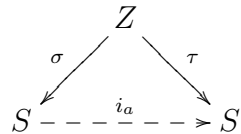
\xymatrix
            {
                 &Z\ar[ld]_\sigma\ar[rd]^\tau&\\
                 S\ar@{-->}[rr]^{i_a}&&S
            }
        $$
        where $\sigma$ is the blow-up of $O_6^a$, $\tau$ is the $G$-equivariant contraction of eight $(-1)$-curves, and $i_a$ is the birational involution $(x:y:z)\mapsto(f_1(x,y,z):f_2(x,y,z):f_3(x,y,z))$, where
        \begin{align}
            f_1(x,y,z)&=\left(a^{12}+1\right) x^2 y^2 z+a^{10} \left(-y^4\right) z+2 a^8 y^2 z^3-a^6 z^5+2 a^4 x^2 z^3-a^2 x^4 z,\nonumber\\
            f_2(x,y,z)&=\left(a^{12}+1\right) x^2 y z^2-a^{10} x^4 y+2 a^8 x^2 y^3-a^6 y^5+2 a^4 y^3 z^2-a^2 y z^4,\text{ and}\nonumber\\
            f_3(x,y,z)&=\left(a^{12}+1\right) x y^2 z^2+a^{10} (-x) z^4+2 a^8 x^3 z^2-a^6 x^5+2 a^4 x^3 y^2-a^2 x y^4.\nonumber
        \end{align}
    \end{itemize}
    We will now show that there is no $G$-link starting from any of the $G$-conic bundles $X_1$, $X_2$, and $X_3$ of degree $5$, except the inverse of the link \eqref{link2}. The $G$-conic bundle $X_1$ is the same as $X$ in the proof of Proposition \ref{sqmf}, and the proof is the same. The $G$-conic bundle $X_2$ is the blow-up of $S$ in the points of $O'_4:=\{(1:\mu_3:\mu_3^2),(-1:\mu_3:\mu_3^2),(1:-\mu_3:\mu_3^2),(1:\mu_3:-\mu_3^2)\}$. The linear system of conics passing through these points is generated by $\mu x^2-z^2=0$ and $(\mu+1)x^2+y^2=0$. The $G$-conic bundle $X_3$ is the blow-up of $S$ in the points of $O'_4:=\{(1:\mu_3^2:\mu_3),(-1:\mu_3^2:\mu_3),(1:-\mu_3^2:\mu_3),(1:\mu_3^2:-\mu_3)\}$. The linear system of conics passing through these points is generated by $\mu x^2-y^2=0$ and $(\mu+1)x^2+z^2=0$. In both of these cases, and the subgroup $N\cong\dz_2^2$ of $G$ acts faithfully on each smooth conic of the system, hence does not fix any point in the fibres of the conic bundle.
\end{proof}

Once again, we get the following consequences.

\begin{corollary}
    Let $G$ be a subgroup of $\aut(S)$ isomorphic to $\aq$. The projective plane is not $G$-solid, but not $G$-birational to any Hirzebruch surface.
\end{corollary}

\begin{corollary}
    Let $G$ be a subgroup of $\aut(S)$ isomorphic to $\aq$. Then $\mathrm{Bir}^G(S)=\left<G,\tau,i_a\vert a\in\dc^*\right>$, where $\tau$ is the standard Cremona involution.
\end{corollary}

Summing up the results of this section, we can conclude about the $G$-solidity of the projective plane.

\begin{theorem}
    Let $G$ be a finite subgroup of $\pgl_3(\dc)$. The following assertions are equivalent.
    \begin{itemize}
        \item The projective plane is $G$-rigid,
        \item The projective plane is $G$-solid,
        \item The group $G$ is transitive and not isomorphic to $\sq$ or $\aq$.
    \end{itemize}
\end{theorem}

\subsection*{Data availability statement}
The author confirms that the data supporting the findings of this study are available within the article and cited references.
\printbibliography
\end{document}